\documentclass{amsart}

\usepackage{amssymb}

\newtheorem{theorem}{Theorem}[section]
\newtheorem{lemma}[theorem]{Lemma}

\theoremstyle{definition}
\newtheorem{definition}[theorem]{Definition}

\theoremstyle{remark}

\numberwithin{equation}{section}

\begin{document}

\title[Whittaker-Hill] {Instability intervals of the Whittaker-Hill operator}

\author{Xu-Dan Luo}
\address{Academy of Mathematics and Systems Science, Chinese Academy of Sciences, Beijing 100190, China.}
\email{lxd@amss.ac.cn}

\subjclass[2010]{Primary : 34L15; 41A60; 47E05.
}

\keywords{The Whittaker-Hill operator; Instability intervals; Asymptotics}

\begin{abstract}
The Hill operator admits a band gap structure. As a special case, like the Mathieu operator, one has only open gaps,
however, the instability intervals of the Whittaker-Hill operator may be open or closed. In 2007, P. Djakov and B. Mityagin gave the asymptotics of band gaps for a special Whittaker-Hill operator [P. Djakov and B. Mityagin, J. Funct. Anal., 242, 157-194 (2007).]. In this paper, a more general Whittaker-Hill operator is considered and the asymptotics of the instability intervals are studied.
\end{abstract}

\maketitle



\section{Introduction and main results}

The Floquet (Bloch) theory indicates that the spectrum of the Schr\"{o}dinger operator
\begin{equation}
\label{E:Schrodinger operator}
Lf:=-f''(x)+\nu(x)f(x), \ \ \ x\in\mathbb{R}
\end{equation}
with a smooth real-valued periodic potential $\nu(x)$ has a band gap structure. If we further assume that $\nu(x)$ is of periodic $\pi$ and set
\begin{equation*}
\nu(x)=-\sum_{n=1}^{\infty}\theta_{n}\cos(2nx)-\sum_{m=1}^{\infty}\phi_{m}\sin (2mx),
\end{equation*}
where $\theta_{n}$ and $\phi_{m}$ are real, then (\ref{E:Schrodinger operator}) can be written as:
\begin{equation}
Lf:=-f''(x)-\left[\sum_{n=1}^{\infty}\theta_{n}\cos(2nx)+\sum_{m=1}^{\infty}\phi_{m}\sin (2mx)\right]f(x).
\end{equation}
Moreover, there are \cite{Eastham} two monotonically increasing infinite sequence of real numbers
\begin{equation*}
\lambda_{0}^{+}, \ \lambda_{1}^{+}, \ \lambda_{2}^{+}, \cdots
\end{equation*}
and
\begin{equation*}
\lambda_{1}^{-}, \ \lambda_{2}^{-}, \ \lambda_{3}^{-}, \cdots
\end{equation*}
such that the Hill equation
\begin{equation}
Lf=\lambda f
\end{equation}
has a solution of period $\pi$ if and only if $\lambda=\lambda_{n}^{+}$, $n=0, 1, 2, \cdots$, and a solution of semi-period $\pi$ (i.e., $f(x+\pi)=-f(x)$) if and only if $\lambda=\lambda_{n}^{-}$, $n=1, 2, 3, \cdots$. The $\lambda_{n}^{+}$ and $\lambda_{n}^{-}$ satisfy the inequalities
\begin{equation*}
\lambda_{0}^{+}<\lambda_{1}^{-}\leq \lambda_{2}^{-}\leq \lambda_{1}^{+}\leq \lambda_{2}^{+}<\lambda_{3}^{-}\leq \lambda_{4}^{-}<\lambda_{3}^{+}\leq \lambda_{4}^{+}<\cdots
\end{equation*}
and the relations
\begin{equation*}
\lim_{n\rightarrow\infty}\lambda_{n}^{+}=\infty, \ \ \ \lim_{n\rightarrow\infty}\lambda_{n}^{-}=\infty.
\end{equation*}
Besides, $\gamma_{n}:=(\lambda_{n+1}^{-}-\lambda_{n}^{-})$ for odd $n$ and $\gamma_{n}:=(\lambda_{n}^{+}-\lambda_{n-1}^{+})$ for even $n$ are referred to as band gaps or instability intervals,
where $n\geq 1$.

It is well-known that there is an extensive theory for the Mathieu operator, where the potential $\nu(x)$ is a single trigonometric function, i.e.,
\begin{equation}
\nu(x)=-B \cos 2x.
\end{equation}

Ince \cite{Ince 3} proved that all instability intervals of the Mathieu operator are open, i.e., no closed gaps for the Mathieu operator. In 1963, Levy and Keller \cite{Levy} gave the asymptotics of $\gamma_{n}=\gamma_{n}(B)$, i.e., for fixed $n$ and real nonzero number $B$, when $B \to 0$,
\begin{equation}
\gamma_{n}=\frac{8}{[(n-1)!]^{2}}\cdot \left(\frac{B}{8}\right)^{n} (1+O(B)).
\end{equation}
18 years later, Harrell \cite{Harrell} gave the asymptotics of the band gaps of the Mathieu operator for fixed $B$ and $n\rightarrow\infty$, i.e.,
\begin{equation}
\gamma_{n}=\lambda_{n}^{+}-\lambda_{n}^{-}=\frac{8}{[(n-1)!]^{2}}\cdot \left(\frac{|B|}{8}\right)^{n}\left(1+O\left(\frac{1}{n^{2}}\right)\right).
\end{equation}

Compared with the Mathieu potential, the band gaps for the Whittaker-Hill potential
\begin{equation}
\label{E:Whittaker-Hill potential}
\nu(x)=-(B\cos 2x+C\cos 4x)
\end{equation}
may be open or closed.
Specifically, if $B=4\alpha t$ and $C=2\alpha^{2}$, for any real $\alpha$ and natural number $t$, it is already known that for odd $t=2m+1$, all the even gaps are closed except the first $m$, but no odd gap disappears; similarly, for even $t=2m$, except for the first $m$, all the odd gaps are closed, but even gaps remain open (see Theorem 11, \cite{Djakov 1} and Theorem 7.9, \cite{Maguns}).

In 2007, P. Djakov and B. Mityagin (see \cite{Djakov 2}) gave the asymptotics of the instability intervals for the above special Whittaker-Hill potential, namely, for real $B, C\neq 0$, $B=4 \alpha t$ and $C=2 \alpha^{2}$, they have the following results, where either both $\alpha$ and $t$ are real numbers if $C>0$ or both $\alpha$ and $t$ are pure imaginary numbers if $C<0$.
\begin{theorem}[\cite{Djakov 2}]
\label{T:Djakov 1}
Let $\gamma_{n}$ be the $n-$th band gap of the Whittaker-Hill operator
\begin{equation}
Lf=-f''-[4\alpha t \cos 2x+ 2\alpha^{2} \cos 4x]f,
\end{equation}
where either both $\alpha$ and $t$ are real, or both are pure imaginary numbers. If $t$ is fixed and $\alpha\rightarrow 0$, then for even $n$
\begin{equation}
\gamma_{n}=\left|\frac{8\alpha^{n}}{2^{n}[(n-1)!]^{2}}\prod_{k=1}^{n/2}(t^{2}-(2k-1)^{2})\right|(1+O(\alpha)),
\end{equation}
and for odd $n$
\begin{equation}
\gamma_{n}=\left|\frac{8\alpha^{n}t}{2^{n}[(n-1)!]^{2}}\prod_{k=1}^{(n-1)/2}(t^{2}-(2k)^{2})\right|(1+O(\alpha)).
\end{equation}
\end{theorem}

\begin{theorem}[\cite{Djakov 2}]
\label{T:Djakov 2}
Let $\gamma_{n}$ be the $n-$th band gap of the Whittaker-Hill operator
\begin{equation}
Lf=-f''-[4\alpha t \cos 2x+ 2\alpha^{2} \cos 4x]f,
\end{equation}
where either both $\alpha$ and $t\neq 0$ are real, or both are pure imaginary numbers. Then the following asymptotic formulae hold for fixed $\alpha$, $t$ and $n\rightarrow\infty$:

for even $n$
\begin{equation}
\gamma_{n}=\frac{8|\alpha|^{n}}{2^{n}[(n-2)!!]^{2}}\left|\cos\left(\frac{\pi}{2}t\right)\right|\left[1+O\left(\frac{\log n}{n}\right)\right],
\end{equation}
and for odd $n$
\begin{equation}
\gamma_{n}=\frac{8|\alpha|^{n}}{2^{n}[(n-2)!!]^{2}}\frac{2}{\pi}\left|\sin\left(\frac{\pi}{2}t\right)\right|\left[1+O\left(\frac{\log n}{n}\right)\right],
\end{equation}
where $(2m-1)!!=1\cdot3\cdots(2m-1)$, \ \ \ $(2m)!!=2\cdot4\cdots(2m)$.
\end{theorem}
In this paper, a more general Whittaker-Hill operator
\begin{equation}
L=-D^{2}+(bq^{m_{1}}\cos 2x+cq^{m_{2}}\cos 4x)
\end{equation}
is considered and the asymptotics of the instability intervals are derived, where $b$, $c$, $q$, $m_{1}$ and $m_{2}$ are real. Our theorems are stated as follows, in particular, we can deduce P. Djakov and B. Mityagin's results by choosing $m_{1}=1$, $m_{2}=2$, $b=-4\alpha t$ and $c=-2\alpha^{2}$.
\begin{theorem}
\label{T:1}
Let the Whittaker-Hill operator be
\begin{equation}
Ly=-y''+(bq^{m_{1}}\cos 2x+cq^{m_{2}}\cos 4x)y,
\end{equation}
where $b$, $c$ and $q$ are real. If $q\rightarrow 0$ and $m_{1}$, $m_{2}$ are positive real parameters, then one of the following results holds:

\begin{enumerate}
\item
when $m_{1}> \frac{m_{2}}{2}$,
\\
(i)
\begin{equation}
\gamma_{2m}=\left|\frac{32\cdot(\frac{c}{2})^{m}\cdot q^{m_{2}m}}{2^{4m}[(m-1)!]^{2}}\right|+O\Big(q^{m_{2}(m-\frac{1}{2})+m_{1}}\Big),
\end{equation}
\\
(ii)
\begin{equation}
\gamma_{1}= |bq^{m_{1}}|+O(q^{2m_{1}-\frac{m_{2}}{2}}),          \ \ \ \gamma_{3}=\frac{|bcq^{m_{1}+m_{2}}|}{8}+O(q^{2m_{1}+\frac{m_{2}}{2}}),
\end{equation}
\begin{equation}
\begin{split}
&\gamma_{2m-1}= \Big|\left(\frac{c}{2}\right)^{m-1}\cdot b\cdot q^{m_{2}(m-1)+m_{1}}\cdot \frac{8}{2^{3m}}\cdot
\Big\{ \frac{1}{[(2m-3)!!]^{2}}\\
&\cdot \sum_{i=1}^{m-2} \frac{(2m-2i-3)!!\cdot (2i-1)!!}{i!\cdot (m-1-i)!} + \frac{2}{(2m-3)!!\cdot(m-1)!} \Big\}\Big|+O\Big(q^{m_{2}(m-\frac{3}{2})+2m_{1}}\Big)
\ \ \ \text{for} \ \ \ m\geq 3;
\end{split}
\end{equation}
\item
when $m_{1}< \frac{m_{2}}{2}$,
\begin{equation}
\gamma_{1}=\left|bq^{m_{1}}\right|+O(q^{m_{2}-m_{1}}), \ \ \ \gamma_{2}=\left|cq^{m_{2}}+\frac{b^{2}q^{2m_{1}}}{8}\right|+O(q^{m_{2}}),
\end{equation}
\begin{equation}
\gamma_{n}=\left|\frac{8\cdot b^{n}\cdot q^{m_{1}n}}{2^{3n}\cdot [(n-1)!]^{2}}\right|+O(q^{m_{1}n+m_{2}-2m_{1}})
\ \ \ \ \ \ \ \text{for} \ \ \ n\geq 3;
\end{equation}
\item
when $m_{1}= \frac{m_{2}}{2}$ and $c<0$,
\\
(i)
\begin{equation}
\gamma_{1}=\left|bq^{m_{1}}\right|+O(q^{3m_{1}}), \ \ \ \gamma_{2}=\left|cq^{m_{2}}+\frac{b^{2}q^{2m_{1}}}{8}\right|+O(q^{4m_{1}}),
\end{equation}
\\
(ii)
\begin{equation}
\gamma_{2m}=8\left|\frac{\prod_{k=1}^{m}\Big(\Big(\frac{b}{2}\Big)^{2}+8c\Big(k-\frac{1}{2}\Big)^{2}\Big)\cdot q^{2m_{1}\cdot m}}{2^{4m}\cdot[(2m-1)!]^{2}}\right|+O(q^{2m_{1}(m+1)}) \ \ \ \ \ \ \ \text{for} \ \ \ m\geq 2,
\end{equation}
\\
(iii)
\begin{equation}
\gamma_{2m-1}=32\left|\frac{\frac{b}{2}\prod_{k=1}^{m-1}\Big(\Big(\frac{b}{2}\Big)^{2}+8ck^{2}\Big)\cdot q^{m_{1}\cdot (2m-1)}}{2^{4m}\cdot[(2m-2)!]^{2}}\right|+O(q^{m_{1}(2m+1)})  \ \ \ \ \ \ \ \text{for} \ \ \ m\geq 2.
\end{equation}
\end{enumerate}
Here, $m$ is a positive integer and $\gamma_{n}$ is the n-th instability interval.
\end{theorem}

\begin{theorem}
\label{T:2}
Let the Whittaker-Hill operator be
\begin{equation}
Ly=-y''+(bq^{m_{1}}\cos 2x+cq^{2m_{1}}\cos 4x)y,
\end{equation}
where $b$, $q$ are real, $c<0$ and $m_{1}>0$. Then the following asymptotic formula holds for fixed $b$, $c$, $q$ and $n\rightarrow\infty$.
\begin{equation}
\gamma_{2m}=\frac{q^{2m_{1}\cdot m}\cdot |c|^{m}}{2^{3m-3}\cdot[(2m-2)!!]^{2}}\cdot \left|\cos\left(\frac{b\pi}{4\sqrt{-2c}}\right)\right|\cdot\left[1+O\left(\frac{\log m}{m}\right)\right],
\end{equation}
\begin{equation}
\gamma_{2m-1}=\frac{q^{m_{1}(2m-1)}\cdot |c|^{m-1}\cdot\sqrt{-2c}}{2^{3m-5}\cdot[(2m-3)!!]^{2}\cdot \pi}
\cdot \left|\sin\left(\frac{b\pi}{4\sqrt{-2c}}\right)\right|\cdot\left[1+O\left(\frac{\log m}{m}\right)\right].
\end{equation}
Here, $m$ is a positive integer and $\gamma_{n}$ is the n-th instability interval.
\end{theorem}

\section{Some lemmas}
\begin{lemma}[\cite{Djakov 2}]
\label{L:1}
Let the Schr\"{o}dinger operator
\begin{equation}
Ly=-y''+v(x)y
\end{equation}
be defined on $\mathbb{R}$, with a real-valued periodic $L^{2}([0, \pi])$-potential $v(x)$, where $v(x)=\sum_{m\in \mathbb{Z}} V(m)\exp(imx)$, $V(m)=0$ for odd $m$, then $\|v\|^{2}=\sum|V(m)|^{2}$.

(a) If $\|v\|<\frac{1}{9}$, then for each $n=1,2,\cdots$, there exists $z=z_{n}$ such that
~\\$|z|\leq 4\|v\|$, and
\begin{equation}
\label{E:length estimation}
2|\beta_{n}(z)|(1-3\|v\|^{2}/n^{2})\leq \gamma_{n}\leq 2|\beta_{n}(z)|(1+3\|v\|^{2}/n^{2}),
\end{equation}
where
\begin{equation}
\beta_{n}(z)=V(2n)+\sum_{k=1}^{\infty}\sum_{j_{1},\cdots,j_{k}\neq \pm n}\frac{V(n-j_{1})V(j_{1}-j_{2})\cdots V(j_{k-1}-j_{k})V(j_{k}+n)}
{(n^{2}-j_{1}^{2}+z)\cdots(n^{2}-j_{k}^{2}+z)}
\end{equation}
converges absolutely and uniformly for $|z|\leq 1$, and $\gamma_{n}$ is the n-th instability interval.

(b) If $V(0)=\frac{1}{\pi}\int_{0}^{\pi}v(x)dx=0$, then there is $N_{0}=N_{0}(v)$ such that (\ref{E:length estimation}) holds for $n\geq N_{0}$ with $z=z_{n}$, $|z_{n}|<1$.

\end{lemma}

\begin{lemma}[\cite{Volkmer}]
\label{L:Volkmer}
The Ince equation
\begin{equation}
\label{E:general Ince}
(1+a\cos 2t)y''(t)+b(\sin 2t)y'(t)+(c+d\cos 2t)y(t)=0
\end{equation}
can be transformed into the Whittaker-Hill equation by assuming
\begin{equation}
a=0,\ \ b=-4q,\ \ c=\lambda+2q^{2}, \ \ d=4(m-1)q.
\end{equation}
Moreover,
\begin{equation}
\label{E:semifinite band gap 1}
\mathrm{sign} (\alpha_{2n}-\beta_{2n})=\mathrm{sign} \ q^{2}\cdot \mathrm{sign} \prod_{p=-n}^{n-1}(2p-m+1)
\end{equation}
and
\begin{equation}
\label{E:semifinite band gap 2}
\mathrm{sign} (\alpha_{2n+1}-\beta_{2n+1})=\mathrm{sign} \ q \cdot \mathrm{sign} \prod_{p=-n}^{n-1}(2p-m),
\end{equation}
where $a$, $b$ and $d$ are real; $\alpha_{2n}$ and $\beta_{2n+2}$ are defined by the eigenvalues corresponding to non-trivial even and odd solutions with period $\pi$, respectively; and $\alpha_{2n+1}$ and $\beta_{2n+1}$ are defined by the eigenvalues corresponding to non-trivial even and odd solutions with semi-period $\pi$, respectively.
\end{lemma}

\begin{lemma}[\cite{Maguns}]
\label{L:Maguns}
The Whittaker-Hill equation
\begin{equation}
f''+[\lambda+4mq\cos 2x+2q^{2}\cos 4x]f=0
\end{equation}
can have two linearly independent solutions of period $\pi$ or $2\pi$ if and only if $m$ is an integer. If $m=2l$ is even, then the odd intervals of instability on the $\lambda$ axis disappear, with at most $|l|+1$ exceptions, but no even interval of instability disappears. If $m=2l+1$ is odd, then at most $|l|+1$ even intervals of instability remain but no interval of instability disappears.
\end{lemma}

\begin{lemma}
\label{L:2}
The Whittaker-Hill operator
\begin{equation}
L=-D^{2}-(B\cos 2x+C\cos 4x)
\end{equation}
admits all even gaps closed except the first $n$ when $\pm\frac{B}{4\sqrt{2C}}=-n-\frac{1}{2}$, $n\in\mathbb{Z}_{\geq 0}$; and all odd gaps closed except the first $n+1$ when $\pm\frac{B}{4\sqrt{2C}}=-n-1$, $n\in\mathbb{Z}_{\geq 0}$.
\end{lemma}

\begin{proof}
By Lemma \ref{L:Maguns}, we obtain that the Whittaker-Hill equation
\begin{equation}
\label{E:W-H-real}
f''(x)+(A+B\cos 2x+C\cos 4x)f(x)=0
\end{equation}
have two linearly independent solutions of period or semi-periodic $\pi$ if and only if $\frac{B}{2\sqrt{2C}}\in \mathbb{Z}$. Moreover, we transform (\ref{E:W-H-real})
into the Ince equation
\begin{equation}
g''(x)+4\sqrt{\frac{C}{2}}\sin 2x\cdot g'(x)+\left[(A+C)+\left(B+4\sqrt{\frac{C}{2}}\right)\cos 2x\right]g(x)=0.
\end{equation}
via $f(x)=e^{-\sqrt{\frac{C}{2}}\cos 2x}\cdot g(x)$. From Lemma \ref{L:Volkmer}, we can write the parameters $q$, $\lambda$ and $m$ of equation (\ref{E:general Ince}) in terms of $A$, $B$ and $C$, i.e.,
\begin{equation*}
q=-\sqrt{\frac{C}{2}}, \ \ \ \lambda=A, \ \ \ m=-\frac{B}{2\sqrt{2C}}.
\end{equation*}

(1) If $m=2n+1$,  $n\in \mathbb{N}^{+}\cup\{0\}$, i.e., $\frac{B}{2\sqrt{2C}}=-2n-1$, and the solutions satisfy the periodic boundary conditions, then we deduce from Lemma \ref{L:Volkmer} that the first $2n+1$ eigenvalues are simple, and others are double.

(2) If $m=2n+2$,  $n\in \mathbb{N}^{+}\cup\{0\}$, i.e., $\frac{B}{2\sqrt{2C}}=-2n-2$, and the solutions satisfy the semi-periodic boundary conditions, then we also deduce from Lemma \ref{L:Volkmer} that the first $2n+2$ eigenvalues are simple, and others are double.

Besides, we can also transform (\ref{E:W-H-real}) into the Ince equation
\begin{equation}
g''(x)-4\sqrt{\frac{C}{2}}\sin 2x\cdot g'(x)+\left[(A+C)+\left(B-4\sqrt{\frac{C}{2}}\right)\cos 2x\right]g(x)=0.
\end{equation}
via $f(x)=e^{\sqrt{\frac{C}{2}}\cos 2x}\cdot g(x)$. Thus, we have similar conclusions.
\end{proof}

In order to prove our results, let us consider all possible walks from $-n$ to $n$. Each such walk is determined by the sequence of its steps
\begin{equation}
x=(x_{1}, \cdots, x_{\nu+1}),
\end{equation}
or by its vertices
\begin{equation}
j_{s}=-n+\sum_{k=1}^{s}x_{k}, \ \ \ s=1, \cdots, \nu.
\end{equation}
The relationship between steps and vertices are given by the formula
\begin{equation}
x_{1}=n+j_{1};\ \ \ x_{k}=j_{k}-j_{k-1}, \ k=2, \cdots, \nu; \ \ \ x_{\nu+1}=n-j_{\nu}.
\end{equation}
\begin{definition}
\label{D:1}
Let $X$ denote the set of all walks from $-n$ to $n$ with steps $\pm 2$ or $\pm 4$. For each $x=(x_{s})_{s=1}^{\nu+1}\in X$ and each $z\in \mathbb{R}$, we define
\begin{equation}
B_{n}(x,z)=\frac{V(x_{1})\cdots V(x_{\nu+1})}{(n^{2}-j_{1}^{2}+z)\cdots(n^{2}-j_{\nu}^{2}+z)}.
\end{equation}
\end{definition}
\begin{definition}
\label{D:2}
Let $X^{+}$ denote the set of all walks from $-n$ to $n$ with positive steps equal to $2$ or $4$. For each $\xi\in X^{+}$, let $X_{\xi}$ denote the set of all walks $x\in X\backslash X^{+}$ such that each vertex of $\xi$ is a vertex of $x$ also. For each $\xi\in X^{+}$ and $\mu\in\mathbb{N}$, let $X_{\xi,\mu}$ be the set of all $x\in X_{\xi}$ such that $x$ has $\mu$ more vertices than $\xi$. Moreover, for each $\mu-$tuple $(i_{1}, \cdots, i_{\mu})$ of integers in $I_{n}=(n+2\mathbb{Z})\setminus \{\pm n\}$, we define $X_{\xi}(i_{1}, \cdots, i_{\mu})$ as the set of all walks $x$ with $\nu+1+\mu$ steps such that $(i_{1}, \cdots, i_{\mu})$ and the sequence of the vertices of $\xi$ are complementary subsequences of the sequence of the vertices of $x$.
\end{definition}
From Definition \ref{D:2}, we deduce
\begin{equation}
X_{\xi,\mu}=\bigcup_{(i_{1}, \cdots, i_{\mu})\in (I_{n})^{\mu}}X_{\xi}(i_{1}, \cdots, i_{\mu}).
\end{equation}

\begin{lemma}[\cite{Djakov 2}]
\label{L:3}
If $\xi\in X^{+}$ and $n\geq 3$, then for $z\in[0,1)$
\begin{equation}
1-z\frac{\log n}{n}\leq \frac{B_{n}(\xi,z)}{B_{n}(\xi,0)}\leq 1-z\frac{\log n}{4n},
\end{equation}
and for $z\in(-1,0]$
\begin{equation}
1+|z|\frac{\log n}{2n}\leq \frac{B_{n}(\xi,z)}{B_{n}(\xi,0)}\leq 1+|z|\frac{2\log n}{n}.
\end{equation}
\end{lemma}

\begin{lemma}[\cite{Djakov 2}]
\label{L:4}
For each walk $\xi\in X^{+}$ and each $\mu-$tuple $(i_{1}, \cdots, i_{\mu})\in (I_{n})^{\mu}$,
\begin{equation}
\sharp X_{\xi}(i_{1}, \cdots, i_{\mu})\leq 5^{\mu}.
\end{equation}
\end{lemma}

\begin{lemma}
\label{L:5}
If $\xi\in X^{+}$ and $|z|\leq 1$, then there exists $n_{1}$ such that for $n\geq n_{1}$,
\begin{equation}
\sum_{x\in X_{\xi}}|B_{n}(x,z)|\leq |B_{n}(\xi,z)|\cdot \frac{K \log n}{n},
\end{equation}
where $K=40 \left(|\frac{b}{2}q^{m_{1}}|+|\frac{c}{2}q^{m_{2}}|+|\frac{b^{2}}{2c}q^{2m_{1}-m_{2}}|\right)$.
\end{lemma}
\begin{proof}
By Definition \ref{D:2}, we have
\begin{equation}
\sum_{x\in X_{\xi}}|B_{n}(x,z)|=\sum_{\mu=1}^{\infty}\sum_{x\in X_{\xi,\mu}}|B_{n}(x,z)|.
\end{equation}
Moreover,
\begin{equation}
\sum_{x\in X_{\xi,\mu}}|B_{n}(x,z)|\leq \sum_{(i_{1}, \cdots, i_{\mu})}\sum_{X_{\xi}(i_{1}, \cdots, i_{\mu})}|B_{n}(x,z)|,
\end{equation}
where the first sum on the right is taken over all $\mu-$tuples $(i_{1}, \cdots, i_{\mu})$ of integers $i_{s}\in n+2\mathbb{Z}$ such that $i_{s}\neq \pm n$.

Fix $(i_{1}, \cdots, i_{\mu})$, if $x\in X_{\xi}(i_{1}, \cdots, i_{\mu})$, then
\begin{equation}
\frac{B_{n}(x,z)}{B_{n}(\xi,z)}=\frac{\prod_{k}V(x_{k})}{\prod_{s}V(\xi_{s})}\cdot \frac{1}{(n^{2}-i_{1}^{2}+z)\cdots(n^{2}-i_{\mu}^{2}+z)}.
\end{equation}
Note that $V(\pm 2)=\frac{b}{2}q^{m_{1}}$ and $V(\pm4)=\frac{c}{2}q^{m_{2}}$.
If each step of $\xi$ is a step of $x$, then
\begin{equation}
\frac{\prod_{k}\left|V(x_{k})\right|}{\prod_{s}\left|V(\xi_{s})\right|}\leq C^{\mu},
\end{equation}
where $C:=|\frac{b}{2}q^{m_{1}}|+|\frac{c}{2}q^{m_{2}}|+|\frac{b^{2}}{2c}q^{2m_{1}-m_{2}}|$. For the general case, let $(j_{s})_{s=1}^{\nu}$ be the vertices of $\xi$, and let us put for convenience $j_{0}=-n$ and $j_{\nu+1}=n$. Since each vertex of $\xi$ is a vertex of $x$, for each $s$, $1\leq s\leq \nu+1$,
\begin{equation}
\xi_{s}=j_{s}-j_{s-1}=\sum_{k\in J_{s}}x_{k},
\end{equation}
where $x_{k}$, $k\in J_{s}$, are the steps of $x$ between the vertices $j_{s-1}$ and $j_{s}$. Fix an $s$, $1\leq s \leq \nu+1$. If $\xi_{s}=2$, then there is a step $x_{k^{*}}$, $k^{*}\in J_{s}$ such that $|x_{k^{*}}|=2$, otherwise, $\xi_{s}$ would be a multiple of $4$. Hence, $|V(\xi_{s})|=|V(x_{k}^{*})|$, which implies that
\begin{equation}
\frac{\prod_{J_{s}}|V(x_{k})|}{|V(\xi_{s})|}\leq C^{b_{s}-1},
\end{equation}
where $b_{s}$ is the cardinality of $J_{s}$.

If $\xi_{s}=4$, there are two possibilities. (1) If there is $k_{*}\in J_{s}$ with $|x_{k_{*}}|=4$, then $|V(\xi_{s})|=|V(x_{k_{*}})|$, so the above inequality holds. (2) There are $k', k''\in J_{s}$ such that $|x_{k'}|=|x_{k''}|=2$, hence,
\begin{equation}
\frac{|V(x_{k'})V(x_{k''})|}{|V(\xi_{s})|}=\left|\frac{b^{2}}{2c}q^{2m_{1}-m_{2}}\right|\leq C,
\end{equation}
so the above inequality also holds. Note that
\begin{equation}
\sum_{s}(b_{s}-1)=\mu,
\end{equation}
we get
\begin{equation}
\frac{\prod_{k}V(x_{k})}{\prod_{s}V(\xi_{s})}\leq C^{\mu}
\end{equation}
holds for the general case.

By
\begin{equation}
\frac{1}{|n^{2}-i^{2}+z|}\leq \frac{2}{|n^{2}-i^{2}|},
\end{equation}
where $i\neq \pm n$, $|z|\leq 1$, we have
\begin{equation}
\frac{|B_{n}(x,z)|}{|B_{n}(\xi,z)|}\leq \frac{(2C)^{\mu}}{|n^{2}-i_{1}^{2}|\cdots |n^{2}-i_{\mu}^{2}|}, \ \ \ x\in X_{\xi}(i_{1}, \cdots, i_{\mu}).
\end{equation}
By Lemma \ref{L:4}, we derive
\begin{equation}
\sum_{x\in X_{\xi}(i_{1}, \cdots, i_{\mu})}\frac{|B_{n}(x,z)|}{|B_{n}(\xi,z)|}\leq \frac{(10C)^{\mu}}{|n^{2}-i_{1}^{2}|\cdots |n^{2}-i_{\mu}^{2}|}.
\end{equation}
Combining Lemma \ref{L:3}, it yields
\begin{equation}
\begin{split}
\sum_{x\in X_{\xi,\mu}}\frac{|B_{n}(x,z)|}{|B_{n}(\xi,z)|}&\leq \sum_{(i_{1}, \cdots, i_{\mu})}\frac{(10C)^{\mu}}{|n^{2}-i_{1}^{2}|\cdots |n^{2}-i_{\mu}^{2}|}\leq\left(\sum_{i\in(n+2\mathbb{Z})\setminus \{\pm n\}} \frac{10C}{|n^{2}-i^{2}|}\right)^{\mu}\\
&\leq (10C)^{\mu}\left(\frac{1+\log n}{n}\right)^{\mu}\leq \left(\frac{20C \log n}{n}\right)^{\mu}.
\end{split}
\end{equation}
Thus,
\begin{equation}
\sum_{x\in X_{\xi,\mu}} |B_{n}(x,z)|\leq |B_{n}(\xi,z)|\cdot \left(\frac{20C \log n}{n}\right)^{\mu}.
\end{equation}
Hence,
\begin{equation}
\sum_{x\in X_{\xi}}\frac{|B_{n}(x,z)|}{|B_{n}(\xi,z)|}\leq \sum_{\mu=1}^{\infty}\left(\frac{20C \log n}{n}\right)^{\mu}.
\end{equation}
We can choose $n_{1}\in \mathbb{N}^{+}$ such that $\frac{20C \log n}{n}\leq \frac{1}{2}$ for $n\geq n_{1}$. Then
\begin{equation}
\sum_{x\in X_{\xi}}\frac{|B_{n}(x,z)|}{|B_{n}(\xi,z)|}\leq \frac{40C \log n}{n}.
\end{equation}
Therefore, there exists $n_{1}$ such that for $n\geq n_{1}$,
\begin{equation}
\sum_{x\in X_{\xi}}|B_{n}(x,z)|\leq |B_{n}(\xi,z)|\cdot \frac{K \log n}{n},
\end{equation}
where $K:=40C=40 \left(|\frac{b}{2}q^{m_{1}}|+|\frac{c}{2}q^{m_{2}}|+|\frac{b^{2}}{2c}q^{2m_{1}-m_{2}}|\right)$.
\end{proof}

\section{Proof of Theorem \ref{T:1}}
Note that
\begin{equation}
V(\pm 2)=\frac{bq^{m_{1}}}{2}, \ \ \ V(\pm 4)=\frac{cq^{m_{2}}}{2}
\end{equation}
and
\begin{equation}
\|v\|^{2}=\frac{1}{2}\Big(b^{2}q^{2m_{1}}+c^{2}q^{2m_{2}}\Big),
\end{equation}
by Lemma \ref{L:1}, we have
\begin{equation}
\gamma_{n}=\pm2\Big(V(2n)+\sum_{k=1}^{\infty} \beta_{k}(n,z)\Big)\Big(1+O(q^{2\cdot\min\{m_{1},m_{2}\}})\Big),
\end{equation}
where
\begin{equation}
\label{E:belta}
\begin{split}
\beta_{k}(n,z)&=\sum_{j_{1},\cdots,j_{k}\neq \pm n}\frac{V(n-j_{1})V(j_{1}-j_{2})\cdots V(j_{k-1}-j_{k})V(j_{k}+n)}
{(n^{2}-j_{1}^{2}+z)\cdots(n^{2}-j_{k}^{2}+z)}\\
&=\sum_{j_{1},\cdots,j_{k}\neq \pm n}\frac{V(n+j_{1})V(j_{2}-j_{1})\cdots V(j_{k}-j_{k-1})V(n-j_{k})}
{(n^{2}-j_{1}^{2}+z)\cdots(n^{2}-j_{k}^{2}+z)}
\end{split}
\end{equation}
and $z=O(q)$. Moreover, all series converge absolutely and uniformly for sufficiently small $q$.

Note that
\begin{equation}
(n+j_{1})+(j_{2}-j_{1})+\cdots+(j_{k}-j_{k-1})+(n-j_{k})=2n,
\end{equation}
and
\begin{equation}
\frac{V(n+j_{1})V(j_{2}-j_{1})\cdots V(j_{k}-j_{k-1})V(n-j_{k})}
{(n^{2}-j_{1}^{2}+z)\cdots(n^{2}-j_{k}^{2}+z)}\neq 0
\end{equation}
when
\begin{equation}
(n+j_{1}),(j_{2}-j_{1}),\cdots,(j_{k}-j_{k-1}),(n-j_{k})\in \{\pm2, \pm4\}.
\end{equation}
We distinguish three cases to discuss.

{\noindent\bf Case 1.} If $m_{1}> \frac{m_{2}}{2}$, then
\begin{equation}
V(n+j_{1})V(j_{2}-j_{1})\cdots V(j_{k}-j_{k-1})V(n-j_{k})
\end{equation}
is a monomial in $q$ of degree at least
\begin{equation}
\frac{m_{2}}{4}\cdot\Big[|n+j_{1}|+|j_{2}-j_{1}|+\cdots+|j_{k}-j_{k-1}|+|n-j_{k}|\Big].
\end{equation}
The minimum case occurs when
\begin{equation}
(n+j_{1}),(j_{2}-j_{1}),\cdots,(j_{k}-j_{k-1}),(n-j_{k})\in \{\pm4\},
\end{equation}
then
\begin{equation}
(n+j_{1})+(j_{2}-j_{1})+\cdots+(j_{k}-j_{k-1})+(n-j_{k})\in 4\mathbb{Z},
\end{equation}
while
\begin{equation}
(n+j_{1})+(j_{2}-j_{1})+\cdots+(j_{k}-j_{k-1})+(n-j_{k})=2n.
\end{equation}
If $n$ is even, i.e., $n=2m$, $m\in\mathbb{Z}_{>0}$, since
\begin{equation}
\begin{split}
&|n+j_{1}|+|j_{2}-j_{1}|+\cdots+|j_{k}-j_{k-1}|+|n-j_{k}|\\
&\geq (n+j_{1})+(j_{2}-j_{1})+\cdots+(j_{k}-j_{k-1})+(n-j_{k})\\
&=2n=4m,
\end{split}
\end{equation}
we obtain
\begin{equation}
V(n+j_{1})V(j_{2}-j_{1})\cdots V(j_{k}-j_{k-1})V(n-j_{k})
\end{equation}
is a monomial in $q$ of degree at least $m_{2}\cdot m$. Such monomial in $q$ of degree $m_{2}\cdot m$ corresponds to a walk from $-n$ to $n$ with vertices $j_{1}, j_{2}, \cdots, j_{k}\neq \pm n$ and positive steps of length $4$. Thus,
\begin{equation}
\gamma_{2m}=\pm P_{2m}(t)q^{m_{2}m}+O\Big(q^{m_{2}(m-\frac{1}{2})+m_{1}}\Big),
\end{equation}
where
\begin{equation}
P_{2m}(t)q^{m_{2}m}=2\Big(V(4m)+\sum_{k=1}^{\infty} \beta_{k}(2m,z)\Big).
\end{equation}
We have
\begin{equation}
P_{2}(t)q^{m_{2}}=2 V(4)=c q^{m_{2}}
\end{equation}
and
\begin{equation}
\begin{split}
&P_{2m}(t)q^{m_{2}m}=2 \sum_{k=1}^{\infty} \beta_{k}(2m,z)\\
&=2 \cdot \Big(\frac{c}{2}\Big)^{m} \cdot q^{m_{2}m}\cdot \prod_{j=1}^{m-1}\Big((4m^{2}-(-2m+4j)^{2})\Big)^{-1}\\
&=\frac{32\cdot(\frac{c}{2})^{m}\cdot q^{m_{2}m}}{2^{4m}[(m-1)!]^{2}}
\end{split}
\end{equation}
for $m\geq 2$. Therefore,
\begin{equation}
\gamma_{2m}=\left|\frac{32\cdot(\frac{c}{2})^{m}\cdot q^{m_{2}m}}{2^{4m}[(m-1)!]^{2}}\right|+O\Big(q^{m_{2}(m-\frac{1}{2})+m_{1}}\Big).
\end{equation}

If $n$ is odd, i.e., $n=2m-1$, $m\in\mathbb{Z}_{>0}$, since
\begin{equation}
(n+j_{1})+(j_{2}-j_{1})+\cdots+(j_{k}-j_{k-1})+(n-j_{k})=2n=4m-2,
\end{equation}
then
\begin{equation}
V(n+j_{1})V(j_{2}-j_{1})\cdots V(j_{k}-j_{k-1})V(n-j_{k})
\end{equation}
is a monomial in $q$ of degree at least
\begin{equation}
\begin{split}
&\frac{m_{2}}{4}\cdot\Big[|n+j_{1}|+|j_{2}-j_{1}|+\cdots+|j_{k}-j_{k-1}|+|n-j_{k}|\Big]+m_{1}-\frac{m_{2}}{2}\\
&\geq \frac{m_{2}}{4}\cdot(4m-2)+m_{1}-\frac{m_{2}}{2}\\
&=m_{2}(m-1)+m_{1}.
\end{split}
\end{equation}
Such monomial in $q$ of degree $m_{2}(m-1)+m_{1}$ corresponds to a walk from $-n$ to $n$ with vertices $j_{1}, j_{2}, \cdots, j_{k}\neq \pm n$ and positive steps. Specifically, except for one step with length $2$, the others are of length $4$. Thus,
\begin{equation}
\gamma_{2m-1}=\pm P_{2m-1}(t)q^{m_{2}(m-1)+m_{1}}+O\Big(q^{m_{2}(m-\frac{3}{2})+2m_{1}}\Big),
\end{equation}
where
\begin{equation}
P_{2m-1}(t)q^{m_{2}(m-1)+m_{1}}=2\Big(V(4m-2)+\sum_{k=1}^{\infty} \beta_{k}(2m-1,z)\Big).
\end{equation}
We obtain
\begin{equation}
P_{1}(t) q^{m_{1}}=2 V(2)=bq^{m_{1}},
\end{equation}
\begin{equation}
P_{3}(t) q^{m_{1}+m_{2}}=2 \sum_{k=1}^{\infty} \beta_{k}(3,z)=2 \left(\frac{bq^{m_{1}}}{2}\right)\left(\frac{cq^{m_{2}}}{2}\right)\left(\frac{1}{3^{2}-1^{2}}+\frac{1}{3^{2}-1^{2}}\right)
=\frac{bcq^{m_{1}+m_{2}}}{8},
\end{equation}
\begin{equation}
\begin{split}
&P_{5}(t) q^{m_{1}+2m_{2}}=2 \sum_{k=1}^{\infty} \beta_{k}(5,z)\\
&=2 \left(\frac{bq^{m_{1}}}{2}\right)\left(\frac{cq^{m_{2}}}{2}\right)^{2}
\left[\frac{1}{(5^{2}-3^{2})(5^{2}-1^{2})}+\frac{1}{(5^{2}-1^{2})(5^{2}-1^{2})}+\frac{1}{(5^{2}-3^{2})(5^{2}-1^{2})}\right]\\
&=\frac{bc^{2}q^{m_{1}+2m_{2}}}{3^{2}\cdot 2^{6}},
\end{split}
\end{equation}
\begin{equation}
\begin{split}
&P_{2m-1}(t)q^{m_{2}(m-1)+m_{1}}=2 \sum_{k=1}^{\infty} \beta_{k}(2m-1,z)\\
&=2 \left(\frac{c}{2}\right)^{m-1}\cdot \left(\frac{b}{2}\right)\cdot q^{m_{2}(m-1)+m_{1}}\cdot
\Big\{\sum_{i=1}^{m-2} \prod_{j=1}^{i}\Big[(2m-1)^{2}-(-2m+1+4j)^{2}\Big]^{-1}\\
&\cdot \prod_{j=i}^{m-2}\Big[(2m-1)^{2}-(-2m+3+4j)^{2}\Big]^{-1}+\prod_{j=0}^{m-2}\Big[(2m-1)^{2}-(-2m+3+4j)^{2}\Big]^{-1}\\
&+\prod_{j=1}^{m-1}\Big[(2m-1)^{2}-(-2m+1+4j)^{2}\Big]^{-1}\Big\}\\
&=2 \left(\frac{c}{2}\right)^{m-1}\cdot \left(\frac{b}{2}\right)\cdot q^{m_{2}(m-1)+m_{1}}\cdot \frac{8}{2^{3m}}\cdot
\Big\{ \frac{1}{[(2m-3)!!]^{2}}\cdot \sum_{i=1}^{m-2} \frac{(2m-2i-3)!!\cdot (2i-1)!!}{i!\cdot (m-1-i)!} \\
&+ \frac{2}{(2m-3)!!\cdot(m-1)!} \Big\}.
\end{split}
\end{equation}
Hence,
\begin{equation}
\gamma_{1}= |bq^{m_{1}}|+O(q^{2m_{1}-\frac{m_{2}}{2}}),          \ \ \ \gamma_{3}=\frac{|bcq^{m_{1}+m_{2}}|}{8}+O(q^{2m_{1}+\frac{m_{2}}{2}}),
\end{equation}
\begin{equation}
\begin{split}
&\gamma_{2m-1}= \Big|\left(\frac{c}{2}\right)^{m-1}\cdot b\cdot q^{m_{2}(m-1)+m_{1}}\cdot \frac{8}{2^{3m}}\cdot
\Big\{ \frac{1}{[(2m-3)!!]^{2}}\\
& \cdot \sum_{i=1}^{m-2} \frac{(2m-2i-3)!!\cdot (2i-1)!!}{i!\cdot (m-1-i)!}
+ \frac{2}{(2m-3)!!\cdot(m-1)!} \Big\}\Big|+O\Big(q^{m_{2}(m-\frac{3}{2})+2m_{1}}\Big)
\end{split}
\end{equation}
for $m\geq 3$.

{\noindent\bf Case 2.} If $m_{1}< \frac{m_{2}}{2}$, then
\begin{equation}
V(n+j_{1})V(j_{2}-j_{1})\cdots V(j_{k}-j_{k-1})V(n-j_{k})
\end{equation}
is a monomial in $q$ of degree at least
\begin{equation}
\frac{m_{1}}{2}\cdot\Big[|n+j_{1}|+|j_{2}-j_{1}|+\cdots+|j_{k}-j_{k-1}|+|n-j_{k}|\Big].
\end{equation}
The minimum case occurs when
\begin{equation}
(n+j_{1}),(j_{2}-j_{1}),\cdots,(j_{k}-j_{k-1}),(n-j_{k})\in \{\pm2\},
\end{equation}
then
\begin{equation}
(n+j_{1})+(j_{2}-j_{1})+\cdots+(j_{k}-j_{k-1})+(n-j_{k})\in 2\mathbb{Z},
\end{equation}
while
\begin{equation}
(n+j_{1})+(j_{2}-j_{1})+\cdots+(j_{k}-j_{k-1})+(n-j_{k})=2n.
\end{equation}
Since
\begin{equation}
\begin{split}
&|n+j_{1}|+|j_{2}-j_{1}|+\cdots+|j_{k}-j_{k-1}|+|n-j_{k}|\\
&\geq (n+j_{1})+(j_{2}-j_{1})+\cdots+(j_{k}-j_{k-1})+(n-j_{k})\\
&=2n,
\end{split}
\end{equation}
we have
\begin{equation}
V(n+j_{1})V(j_{2}-j_{1})\cdots V(j_{k}-j_{k-1})V(n-j_{k})
\end{equation}
is a monomial in $q$ of degree at least $m_{1}\cdot n$. Such monomial in $q$ of degree $m_{1}\cdot n$ corresponds to a walk from $-n$ to $n$ with vertices $j_{1}, j_{2}, \cdots, j_{k}\neq \pm n$ and positive steps of length $2$. Thus,
\begin{equation}
\gamma_{n}=\pm P_{n}(t) q^{m_{1}n}+O(q^{m_{1}n+m_{2}-2m_{1}}),
\end{equation}
where
\begin{equation}
P_{n}(t) q^{m_{1}n}=2\Big(V(2n)+\sum_{k=1}^{\infty} \beta_{k}(n,z)\Big).
\end{equation}
We deduce
\begin{equation}
P_{1}(t)q^{m_{1}}=2V(2)=bq^{m_{1}}, \ \ \ P_{2}(t)q^{2m_{1}}=2 \left(V(4)+\frac{\Big(\frac{bq^{m_{1}}}{2}\Big)^{2}}{2^{2}}\right)=cq^{m_{2}}+\frac{b^{2}q^{2m_{1}}}{8},
\end{equation}
\begin{equation}
\begin{split}
&P_{n}(t) q^{m_{1}n}=2\sum_{k=1}^{\infty} \beta_{k}(n,z)\\
&=2\cdot\Big(\frac{b}{2}\Big)^{n}\cdot q^{m_{1}\cdot n}\cdot \prod_{j=1}^{n-1}(n^{2}-(-n+2j)^{2})^{-1}\\
&=2\cdot \Big(\frac{b}{2}\Big)^{n}\cdot q^{m_{1}\cdot n}\cdot \prod_{j=1}^{n-1}(n^{2}-(-n+2j)^{2})^{-1}\\
&=2\cdot \frac{(\frac{b}{2})^{n}\cdot q^{m_{1}\cdot n}}{4^{n-1}\cdot[(n-1)!]^{2}}\\
&=\frac{8\cdot b^{n}\cdot q^{m_{1}n}}{2^{3n}\cdot [(n-1)!]^{2}}
\end{split}
\end{equation}
for $n\geq 3$. Therefore,
\begin{equation}
\gamma_{1}=\left|bq^{m_{1}}\right|+O(q^{m_{2}-m_{1}}), \ \ \ \gamma_{2}=\left|cq^{m_{2}}+\frac{b^{2}q^{2m_{1}}}{8}\right|+O(q^{m_{2}}),
\end{equation}
\begin{equation}
\gamma_{n}=\left|\frac{8\cdot b^{n}\cdot q^{m_{1}n}}{2^{3n}\cdot [(n-1)!]^{2}}\right|+O(q^{m_{1}n+m_{2}-2m_{1}})
\end{equation}
for $n\geq 3$.

{\noindent\bf Case 3.} If $m_{1}= \frac{m_{2}}{2}$, then
\begin{equation}
V(n+j_{1})V(j_{2}-j_{1})\cdots V(j_{k}-j_{k-1})V(n-j_{k})
\end{equation}
is a monomial in $q$ of degree
\begin{equation}
\frac{m_{1}}{2}\cdot\Big[|n+j_{1}|+|j_{2}-j_{1}|+\cdots+|j_{k}-j_{k-1}|+|n-j_{k}|\Big].
\end{equation}
Since
\begin{equation}
\begin{split}
&|n+j_{1}|+|j_{2}-j_{1}|+\cdots+|j_{k}-j_{k-1}|+|n-j_{k}|\\
&\geq (n+j_{1})+(j_{2}-j_{1})+\cdots+(j_{k}-j_{k-1})+(n-j_{k})\\
&=2n,
\end{split}
\end{equation}
we have
\begin{equation}
V(n+j_{1})V(j_{2}-j_{1})\cdots V(j_{k}-j_{k-1})V(n-j_{k})
\end{equation}
is a monomial in $q$ of degree at least $m_{1}\cdot n$, and each such monomial of degree $m_{1}\cdot n$ corresponds to a walk from $-n$ to $n$ with vertices $j_{1}, j_{2}, \cdots, j_{k}\neq \pm n$ and positive steps of length $2$ or $4$. The minimum case occurs when $n+j_{1}$, $j_{2}-j_{1}$, $\cdots$, $j_{k}-j_{k-1}$ and $n-j_{k}$ are of the same sign,
while the second smallest degree is for one step of length $2$ with opposite sign. Thus,
\begin{equation}
\gamma_{n}=\pm P_{n}(t) q^{m_{1}n}+O(q^{m_{1}(n+2)}),
\end{equation}
where
\begin{equation}
P_{n}(t) q^{m_{1}n}=2\Big(V(2n)+\sum_{k=1}^{\infty} \beta_{k}(n,z)\Big).
\end{equation}
We obtain
\begin{equation}
P_{1}(t)q^{m_{1}}=2V(2)=bq^{m_{1}}, \ \ \ P_{2}(t)q^{2m_{1}}=2 \left(V(4)+\frac{\Big(\frac{bq^{m_{1}}}{2}\Big)^{2}}{2^{2}}\right)=cq^{m_{2}}+\frac{b^{2}q^{2m_{1}}}{8},
\end{equation}
\begin{equation}
\begin{split}
&P_{n}(t) q^{m_{1}n}=2\sum_{k=1}^{\infty} \beta_{k}(n,z)\\
&=2\cdot P_{n}\Big(\frac{b}{2}\Big)\cdot q^{m_{1}\cdot n}\cdot \prod_{j=1}^{n-1}(n^{2}-(-n+2j)^{2})^{-1}\\
&=2\cdot P_{n}\Big(\frac{b}{2}\Big)\cdot q^{m_{1}\cdot n}\cdot \prod_{j=1}^{n-1}(n^{2}-(-n+2j)^{2})^{-1}\\
&=8\cdot \frac{P_{n}\big(\frac{b}{2}\big)\cdot q^{m_{1}\cdot n}}{2^{2n}\cdot[(n-1)!]^{2}}
\end{split}
\end{equation}
for $n\geq 3$. Therefore,
\begin{equation}
\gamma_{1}=\left|bq^{m_{1}}\right|+O(q^{3m_{1}}), \ \ \ \gamma_{2}=\left|cq^{m_{2}}+\frac{b^{2}q^{2m_{1}}}{8}\right|+O(q^{4m_{1}}),
\end{equation}
\begin{equation}
\gamma_{n}=\left|8\cdot \frac{P_{n}\big(\frac{b}{2}\big)\cdot q^{m_{1}\cdot n}}{2^{2n}\cdot[(n-1)!]^{2}}\right|+O(q^{m_{1}(n+2)})
\end{equation}
for $n\geq 3$, where $P_{n}\big(\frac{b}{2}\big)$ is a polynomial of $\frac{b}{2}$ with degree $n$ and leading coefficient $1$.

Specifically, if $n$ is even, i.e., $n=2m$, $m\in\mathbb{Z}_{>0}$, then
\begin{equation}
(n+j_{1})+(j_{2}-j_{1})+\cdots+(j_{k}-j_{k-1})+(n-j_{k})=4m,
\end{equation}
which implies that each walk from $-2m$ to $2m$ has even number of steps with length $2$. We have
\begin{equation}
P_{2m}\Big(\frac{b}{2}\Big)=\prod_{k=1}^{m}\Big(\Big(\frac{b}{2}\Big)^{2}-x_{k}\Big),
\end{equation}
where $x_{k}$, $k=1,\cdots, m$, depend on $m$. By Lemma \ref{L:2}, we obtain all even gaps closed except the first $k$ if $\big(\frac{b}{2}\big)=-8c\big(k+\frac{1}{2}\big)^{2}$, which yields
\begin{equation}
P_{2m}\Big(\frac{b}{2}\Big)=\prod_{k=1}^{m}\Big(\Big(\frac{b}{2}\Big)^{2}+8c\Big(k-\frac{1}{2}\Big)^{2}\Big).
\end{equation}
Hence,
\begin{equation}
\gamma_{2m}=8\left|\frac{\prod_{k=1}^{m}\Big(\Big(\frac{b}{2}\Big)^{2}+8c\Big(k-\frac{1}{2}\Big)^{2}\Big)\cdot q^{2m_{1}\cdot m}}{2^{4m}\cdot[(2m-1)!]^{2}}\right|+O(q^{2m_{1}(m+1)})
\end{equation}
for $m\geq 2$. If $n$ is odd, i.e., $n=2m-1$, $m\in\mathbb{Z}_{>0}$, then
\begin{equation}
(n+j_{1})+(j_{2}-j_{1})+\cdots+(j_{k}-j_{k-1})+(n-j_{k})=2n=4m-2,
\end{equation}
which implies that each walk from $-2m$ to $2m$ has odd number of steps with length $2$. We have
\begin{equation}
P_{2m-1}\Big(\frac{b}{2}\Big)=\frac{b}{2}\prod_{k=1}^{m-1}\Big(\Big(\frac{b}{2}\Big)^{2}-y_{k}\Big),
\end{equation}
where $y_{k}$, $k=1,\cdots, m-1$, depend on $m$. By Lemma \ref{L:2}, we deduce
\begin{equation}
P_{2m-1}\Big(\frac{b}{2}\Big)=\frac{b}{2}\prod_{k=1}^{m-1}\Big(\Big(\frac{b}{2}\Big)^{2}+8ck^{2}\Big).
\end{equation}
Hence,
\begin{equation}
\gamma_{2m-1}=32\left|\frac{\frac{b}{2}\prod_{k=1}^{m-1}\Big(\Big(\frac{b}{2}\Big)^{2}+8ck^{2}\Big)\cdot q^{m_{1}\cdot (2m-1)}}{2^{4m}\cdot[(2m-2)!]^{2}}\right|+O(q^{m_{1}(2m+1)})
\end{equation}
for $m\geq 2$.

\section{Proof of Theorem \ref{T:2}}
Since $V(\pm 2)=\frac{b}{2}q^{m_{1}}$ and $V(\pm 4)=\frac{c}{2}q^{2m_{1}}$, thus,
\begin{equation}
\|v\|^{2}=\frac{1}{2}\Big(b^{2}q^{2m_{1}}+c^{2}q^{4m_{1}}\Big).
\end{equation}
By Lemma \ref{L:1}, we get
\begin{equation}
\gamma_{n}=2\left|\sum_{x\in X}B_{n}(x,z)\right|\left(1+O\left(\frac{1}{n^{2}}\right)\right),
\end{equation}
where $z=z_{n}$ depends on $n$, but $|z|<1$.

Set $\sigma_{n}=\sum_{\xi\in X^{+}}B_{n}(\xi,0):=\sigma_{n}^{+}+\sigma_{n}^{-}$, where
$\sigma_{n}^{\pm}:=\sum_{\xi:B_{n}(\xi,0)\gtrless 0}B_{n}(\xi,0)$.
When $\xi\in X^{+}$,
\begin{equation}
B_{n}(\xi,0)=\frac{V(x_{1})\cdots V(x_{\nu+1})}{(n^{2}-j_{1}^{2})\cdots(n^{2}-j_{\nu}^{2})},
\end{equation}
where $x_{i}=2$ or $4$ for $i=1, \cdots, \nu+1$.

Note that $X\setminus X^{+}= \bigcup_{\xi\in X^{+}} X_{\xi}$, we choose disjoint sets $X_{\xi}'\subset X_{\xi}$ so that
\begin{equation}
X\setminus X^{+}=\bigcup_{\xi\in X^{+}}X_{\xi}'.
\end{equation}
Then
\begin{equation}
\sum_{x\in X\setminus X^{+}}B_{n}(x,z)=\sum_{\xi\in X^{+}}\left(\sum_{x\in X_{\xi}'}B_{n}(x,z)\right),
\end{equation}
therefore, we have
\begin{equation}
\begin{split}
&\sum_{x\in X}B_{n}(x,z)=\sum_{\xi\in X^{+}}\left(B_{n}(\xi,z)+\sum_{x\in X_{\xi}'}B_{n}(x,z)\right)\\
&=\sum_{\xi:B_{n}(\xi,0)>0}\left(B_{n}(\xi,z)+\sum_{x\in X_{\xi}'}B_{n}(x,z)\right)
+\sum_{\xi:B_{n}(\xi,0)<0}\left(B_{n}(\xi,z)+\sum_{x\in X_{\xi}'}B_{n}(x,z)\right)\\
&:\Sigma=\Sigma^{+}+\Sigma^{-},
\end{split}
\end{equation}
where $\Sigma^{\pm}:=\sum_{\xi:B_{n}(\xi,0)\gtrless 0}\left(B_{n}(\xi,z)+\sum_{x\in X_{\xi}'}B_{n}(x,z)\right)$.

By Lemma \ref{L:3} and Lemma \ref{L:5}, we get there exists a constant $C_{1}>0$ such that
\begin{equation}
\left[1\mp C_{1}\frac{\log n}{n}\right]\sigma_{n}^{\pm}\leq\Sigma^{\pm}\leq \left[1\pm C_{1}\frac{\log n}{n}\right]\sigma_{n}^{\pm},
\end{equation}
which is followed by
\begin{equation}
\label{E:estimation}
\left|\frac{\Sigma}{\sigma_{n}}-1\right|\leq C_{1} \frac{|\sigma_{n}^{-}|+\sigma_{n}^{+}}{|\sigma_{n}|} \cdot \frac{\log n}{n}.
\end{equation}

If $\xi\in X^{+}$, then $V(x_{1})\cdots V(x_{\nu+1})$ is a monomial in $q$ of degree $\frac{m_{1}}{2}\cdot (x_{1}+\cdots+x_{\nu+1})=m_{1}\cdot n$. From Case 3 of Theorem \ref{T:1}, we have
\begin{equation}
\sigma_{2m}=\sum_{\xi\in X^{+}}B_{2m}(\xi,0)=\frac{q^{2m_{1}\cdot m}}{4^{2m-1}\cdot[(2m-1)!]^{2}}\cdot\prod_{k=1}^{m}\left(\left(\frac{b}{2}\right)^{2}+8c\left(k-\frac{1}{2}\right)^{2}\right)
\end{equation}
and
\begin{equation}
\sigma_{2m-1}=\sum_{\xi\in X^{+}}B_{2m-1}(\xi,0)=\frac{q^{m_{1}(2m-1)}}{4^{2m-2}\cdot[(2m-2)!]^{2}}\cdot\frac{b}{2}
\cdot\prod_{k=1}^{m-1}\left(\left(\frac{b}{2}\right)^{2}+8ck^{2}\right).
\end{equation}
Moreover, $\sigma_{2m}\neq 0$ when $\frac{b}{2}\neq 2\sqrt{-2c}\cdot (k-\frac{1}{2})$ and $\sigma_{2m-1}\neq 0$
when $\frac{b}{2}\neq 2\sqrt{-2c} \cdot k$, where $c<0$. So
\begin{equation}
\label{E:upper bound 1}
\frac{|\sigma_{2m}^{-}|+\sigma_{2m}^{+}}{|\sigma_{2m}|}
=\frac{\prod_{k=1}^{m}\left(1-\frac{b^{2}}{8c(2k-1)^{2}}\right)}{\prod_{k=1}^{m}\left|1+\frac{b^{2}}{8c(2k-1)^{2}}\right|}
\leq \frac{\prod_{k=1}^{\infty}\left(1-\frac{b^{2}}{8c(2k-1)^{2}}\right)}{\prod_{k=1}^{\infty}\left|1+\frac{b^{2}}{8c(2k-1)^{2}}\right|}
=\left|\frac{\cosh \left(\frac{b\pi}{4\sqrt{-2c}}\right)}{\cos \left(\frac{b\pi}{4\sqrt{-2c}}\right)}\right|.
\end{equation}
Similarly, we have
\begin{equation}
\label{E:upper bound 2}
\frac{|\sigma_{2m-1}^{-}|+\sigma_{2m-1}^{+}}{|\sigma_{2m-1}|}\leq \left|\frac{\sinh \left(\frac{b\pi}{4\sqrt{-2c}}\right)}{\sin \left(\frac{b\pi}{4\sqrt{-2c}}\right)}\right|.
\end{equation}
By (\ref{E:estimation}), we obtain
\begin{equation}
\sum_{x\in X}B_{2m}(x,z)=\sigma_{2m}\left[1+O\left(\frac{\log m}{m}\right)\right]=\left(\sum_{\xi\in X^{+}}B_{2m}(\xi,0)\right)\left[1+O\left(\frac{\log m}{m}\right)\right]
\end{equation}
and
\begin{equation}
\sum_{x\in X}B_{2m-1}(x,z)=\sigma_{2m-1}\left[1+O\left(\frac{\log m}{m}\right)\right]=\left(\sum_{\xi\in X^{+}}B_{2m-1}(\xi,0)\right)\left[1+O\left(\frac{\log m}{m}\right)\right].
\end{equation}

Notice that
\begin{equation}
\cos \left(\frac{b\pi}{4\sqrt{-2c}}\right)=\prod_{k=1}^{\infty}\left(1+\frac{b^{2}}{8c(2k-1)^{2}}\right)
\end{equation}
and
\begin{equation}
\sin \left(\frac{b\pi}{4\sqrt{-2c}}\right)=\frac{b\pi}{4\sqrt{-2c}}\prod_{k=1}^{\infty}\left(1+\frac{b^{2}}{8c(2k)^{2}}\right),
\end{equation}
then
\begin{equation}
\cos \left(\frac{b\pi}{4\sqrt{-2c}}\right)=\prod_{k=1}^{m}\left(1+\frac{b^{2}}{8c(2k-1)^{2}}\right)\left[1+O\left(\frac{1}{m}\right)\right]
\end{equation}
and
\begin{equation}
\sin \left(\frac{b\pi}{4\sqrt{-2c}}\right)=\frac{b\pi}{4\sqrt{-2c}}\prod_{k=1}^{m-1}\left(1+\frac{b^{2}}{8c(2k)^{2}}\right)\left[1+O\left(\frac{1}{m}\right)\right].
\end{equation}
Hence,
\begin{equation}
\sum_{\xi\in X^{+}}B_{2m}(\xi,0)=\frac{q^{2m_{1}\cdot m}\cdot(-1)^{m}\cdot c^{m}}{2^{3m-2}\cdot[(2m-2)!!]^{2}}\cdot \cos\left(\frac{b\pi}{4\sqrt{-2c}}\right)\cdot \left[1+O\left(\frac{1}{m}\right)\right]
\end{equation}
and
\begin{equation}
\sum_{\xi\in X^{+}}B_{2m-1}(\xi,0)=\frac{q^{m_{1}(2m-1)}\cdot(-1)^{m-1}\cdot c^{m-1}\cdot\sqrt{-2c}}{2^{3m-4}\cdot[(2m-3)!!]^{2}\cdot \pi}
\cdot \sin\left(\frac{b\pi}{4\sqrt{-2c}}\right)\cdot \left[1+O\left(\frac{1}{m}\right)\right].
\end{equation}
Combining (\ref{E:estimation}), (\ref{E:upper bound 1}) and (\ref{E:upper bound 2}), we deduce
\begin{equation}
\begin{split}
&\sum_{x\in X}B_{2m}(x,z)=\left(\sum_{\xi\in X^{+}}B_{2m}(\xi,0)\right)\left[1+O\left(\frac{\log m}{m}\right)\right]\\
&=\frac{q^{2m_{1}\cdot m}\cdot(-1)^{m}\cdot c^{m}}{2^{3m-2}\cdot[(2m-2)!!]^{2}}\cdot \cos\left(\frac{b\pi}{4\sqrt{-2c}}\right)\cdot\left[1+O\left(\frac{\log m}{m}\right)\right]
\end{split}
\end{equation}
and
\begin{equation}
\begin{split}
&\sum_{x\in X}B_{2m-1}(x,z)=\left(\sum_{\xi\in X^{+}}B_{2m-1}(\xi,0)\right)\left[1+O\left(\frac{\log m}{m}\right)\right]\\
&=\frac{q^{m_{1}(2m-1)}\cdot(-1)^{m-1}\cdot c^{m-1}\cdot\sqrt{-2c}}{2^{3m-4}\cdot[(2m-3)!!]^{2}\cdot \pi}
\cdot \sin\left(\frac{b\pi}{4\sqrt{-2c}}\right)\cdot\left[1+O\left(\frac{\log m}{m}\right)\right].
\end{split}
\end{equation}
Therefore,
\begin{equation}
\gamma_{2m}=\frac{q^{2m_{1}\cdot m}\cdot |c|^{m}}{2^{3m-3}\cdot[(2m-2)!!]^{2}}\cdot \left|\cos\left(\frac{b\pi}{4\sqrt{-2c}}\right)\right|\cdot\left[1+O\left(\frac{\log m}{m}\right)\right]
\end{equation}
and
\begin{equation}
\gamma_{2m-1}=\frac{q^{m_{1}(2m-1)}\cdot |c|^{m-1}\cdot\sqrt{-2c}}{2^{3m-5}\cdot[(2m-3)!!]^{2}\cdot \pi}
\cdot \left|\sin\left(\frac{b\pi}{4\sqrt{-2c}}\right)\right|\cdot\left[1+O\left(\frac{\log m}{m}\right)\right].
\end{equation}

\bibliographystyle{amsplain}

\begin{thebibliography}{10}

















\bibitem{Djakov 1} P. Djakov and B. Mityagin, \textsl{Simple and double eigenvalues of the Hill operator with a two-term potential},
Journal of Approximation Theory \textbf{135}, 70-104 (2005).

\bibitem{Djakov 2} P. Djakov and B. Mityagin, \textsl{Asymptotics of instability zones of the Hill operator
with a two term potential}, Journal of Functional Analysis \textbf{242}, 157-194 (2007).

\bibitem{Eastham} M.S.P. Eastham, \textsl{The spectral theory of periodic differential equations}, Scottish academic press, 1973.





\bibitem{Harrell} E. Harrell, \textsl{On the effect of the boundary conditions on the eigenvalues of ordinary differential equations},
in: Suppl. Amer. J. Math., dedicated to P. Hartman, Johns Hopkins Univ. Press, Baltimore, MD,  139-150 (1981).










\bibitem{Ince 3} E. L. Ince, \textsl{A proof of the impossibility of the coexistence of two Mathieu functions}, Proc. Cambridge Philos.
Soc. \textbf{21} 117-120 (1922).








\bibitem{Levy} D. M. Levy and J. B. Keller, \textsl{Instability Intervals of Hill's Equation}, Communications on Pure and Applied
Mathematics, \textbf{XVI} 469-476 (1963).




\bibitem{Maguns} W. Magnus and Winkler, \textsl{Hill's equation}, InterScience publisher, 1966.
















\bibitem{Volkmer} H. Volkmer, \textsl{Coexistence of periodic solutions of Ince's equation}, Analysis(Munich) \textbf{23}, 97-105 (2003).





\end{thebibliography}

\end{document}